\documentclass[11pt]{article}

\usepackage{fullpage}
\usepackage{amsmath}
\usepackage{amssymb}
\usepackage{amsthm}
\usepackage{graphicx}
\usepackage{color}
\usepackage{float}
\usepackage{caption}
\usepackage{bbm}

\newtheorem{obs}{Observation}
\newtheorem{prop}{Proposition}
\newtheorem{thm}{Theorem}
  
\DeclareMathOperator*{\supp}{supp}

\DeclareMathOperator*{\argmax}{argmax}

\begin{document}

\title{Deconvolution with a Box}
\author{Pedro F. Felzenszwalb \\
  Brown University \\
  {\tt pff@brown.edu}}
\date{}
\maketitle

\maketitle

Abstract: Deconvolution with a box (square wave) is a key operation
for super-resolution with pixel-shift cameras.  In general convolution
with a box is not invertible.  However, we can obtain perfect
reconstructions of sparse signals using convex optimization.  We give a
direct proof that improves on the reconstruction bound that follows
from the general results in \cite{Donoho1989}.  We also show our bound
is tight and matches an information theoretic limit.

In a pixel-shift camera we can move a low-resolution sensor in a grid
of high-resolution locations.  By taking a series of pictures and
re-arranging the data we obtain the result of convolving a
high-resolution target with a box.  Our goal is to recover the
high-resolution target.  See Figure~\ref{fig:result} for a numerical
experiment simulating this process.

\begin{figure}[h]
  \centering
  \begin{tabular}{ccc}
  \includegraphics[scale=0.8]{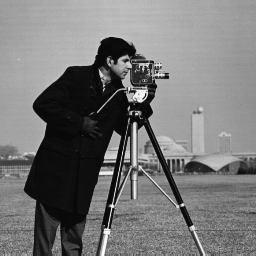} &
  \includegraphics[scale=0.8]{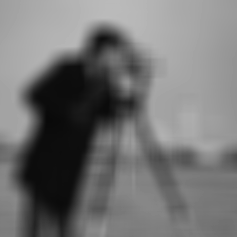} &
  \includegraphics[scale=0.8]{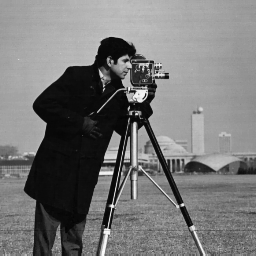} \\
  (a) & (b) & (c)
  \end{tabular}
  \caption{(a) 256x256 high-resolution target; (b) ``valid'' part of the
    convolution of the target with a box of width 20, obtained by
    ``scanning'' the target using a low-resolution sensor (each sensor
    pixel is 20 times the size of the target pixel size); (c)
    reconstruction using convex optimization.}
  \label{fig:result}
\end{figure}

Let $x \in \mathbb{R}^n$ be a discrete (unknown) signal of length $n$.  

Let $\supp(x) \subseteq \{1,\ldots,n\}$ denote the indices where $x$ is nonzero.

Let $y = h \otimes x$ where $h$ is a box of width $k$ and $\otimes$
denotes either regular or circular convolution.  Equivalently $y=Ax$
or $y=Bx$ where $A$ and $B$ are the Toeplitz matrices defined below.

Convolution with a box of width $k$ computes sums of $k$ consecutive
values in $x$.  The ``valid'' part (defined without zero padding) of the
regular convolution $h \otimes x$ is given by an $(n-k+1) \times n$
matrix.  When $n=6$ and $k=3$ we have
$$A =
\begin{pmatrix}
1 & 1 & 1 & 0 & 0 & 0 \\
0 & 1 & 1 & 1 & 0 & 0 \\
0 & 0 & 1 & 1 & 1 & 0 \\
0 & 0 & 0 & 1 & 1 & 1 
\end{pmatrix}.$$

In the circular case we have an $n \times n$ matrix obtained from $A$ by adding $k-1$ rows,
$$B =
\begin{pmatrix}
1 & 1 & 1 & 0 & 0 & 0 \\
0 & 1 & 1 & 1 & 0 & 0 \\
0 & 0 & 1 & 1 & 1 & 0 \\
0 & 0 & 0 & 1 & 1 & 1 \\
1 & 0 & 0 & 0 & 1 & 1 \\
1 & 1 & 0 & 0 & 0 & 1 
\end{pmatrix}.$$

We would like to recover $x$ from $y=Ax$ (or $y=Bx$).  We show that
when $x$ is sparse we can recover $x$ by minimizing $||x||_1$ subject
to the constraints imposed by $y$.

In practice many signals such as ``natural'' images are not sparse but
are piecewise constant and have sparse derivatives.  In the one
dimensional case we can minimize $||Dx||_1$ where $D$ is a finite
difference operator (see \cite{Candes06}).  For the two-dimensional
reconstruction in Figure~\ref{fig:result} we select a solution $x$
minimizing a weighted sum $||Ax-y||_2^2 + \lambda ||\nabla x||_1$.
This example illustrates the approach leads to stable reconstructions
even when we increase the resolution by large factor. This is in
contrast to the theoretical and practical limits of super-resolution
described in \cite{Baker02} and \cite{Lin04}.

Our approach does not rely on frequency domain concepts.  Instead we
characterize the kernel of $A$ by recalling that convolution with a
box can be implemented by a recursive filter,
$$(Ax)_{i+1} =
  (Ax)_{i} + x_{i+k} - x_{i}.$$

\begin{prop}
  $z \in \ker(A)$ if and only if $z_j = z_{j \bmod k}$ and $\sum_{j=1}^k
  z_j = 0$.
\end{prop}
\begin{proof}
  Suppose $z \in \ker(A)$.  For any signal $x$ we have $(Ax)_{i+1} =
  (Ax)_{i} + x_{i+k} - x_{i}$.  Since $(Az)_i = 0$ for all $i$ this
  implies $z_{j+k} = z_j$.  Since $(Az)_1 = 0$ we have $\sum_{j=1}^k
  z_j = 0$.  For the other direction, note that if $z_j = z_{j \bmod
    k}$ then any $k$ consecutive terms of $z$ equal to
  $(z_1,\ldots,z_k)$ after a cyclic shift.  Since the sum of those
  terms is zero we have $z \in \ker(A)$.
\end{proof}

\begin{obs}
  $\ker(B) \subseteq \ker(A)$.
\end{obs}

\begin{obs}
  When $k$ divides $n$ $\ker(B) = \ker(A)$.
\end{obs}

The first observation follows from the fact that $B$ can be obtained from $A$ by adding rows.
The second observation follows from the characterization of $\ker(A)$ and and the first observation.

\begin{obs}
  $\dim(\ker(A)) = k-1$.
\end{obs}

For $x \in \mathbb{R}^n$ and $S \subseteq \{1,\ldots,n\}$ let $x_S$ be
a vector that equals $x$ in the entries indexed by $S$ and zero in the
complement of $S$.  Note that $||w||_1 = ||w_S||_1 +
||w_{\bar{S}}||_1$.

\begin{thm}
  If $Ax=y$ ($Bx=y$) and $|\supp(x)| < \lfloor \frac{n}{k} \rfloor$ then $x$ is the
  unique solution to $Ax=y$ ($Bx=y$) minimizing $||x||_1$.
\end{thm}

\begin{proof}
  Suppose $Ax'=y$ with $x' \neq x$.  We show $||x'||_1 > ||x||_1$.

  $x' = x+z$ with $z \in \ker(A)$ and $z \neq 0$.

  We first show that if $|S| < \lfloor n/k \rfloor$ then
  $||z_{\bar{S}}||_1 > ||z_S||_1$ (the nullspace property).

  Recall that $z_j = z_{j \bmod k}$ and $\sum_{j=1}^k z_j = 0$.

  Let $[k] = \{1,\ldots,k\}$.

  Let $i = \argmax_{j \in [k]} |z_j|$.

  Since $|S| < \lfloor n/k \rfloor$ we can maximize $||z_S||_1$ and minimize
  $||z_{\bar{S}}||_1$ simultaneously by selecting $|S|$ indices that equal
  $i$ modulo $k$ to include in $S$.  In this case,
  $$||z_S||_1 = |S| |z_i|.$$
  $$||z_{\bar{S}}||_1 \ge |S| \left(\sum_{j\in [k] \setminus i} |z_j|\right) + \sum_{j \in [k]} |z_j|.$$

  Note that $\sum_{j=1}^k z_j = 0$ implies $-z_i = \sum_{j \in [k] \setminus i} z_j$
  and therefore $|z_i| \le \sum_{j \in [k] \setminus i} |z_j|$.

  Since $z \neq 0$ we know $\sum_{j \in [k]} |z_j| > 0$.
  We conclude $||z_{\bar{S}}||_1 > ||z_S||_1$.

  The remainder of the proof is a standard argument.

  Let $S = \supp(x)$.  Then
  $$||x'||_1 = ||x+z||_1 = ||(x+z)_S||_1 +
  ||(x+z)_{\bar{S}}||_1 = ||x_S + z_S||_1 + ||z_{\bar{S}}||_1.$$
  By the triangle inequality,
  $$||x_S + z_S||_1 \ge  ||x_S||_1 - ||z_S||_1.$$
  Therefore
  $$||x'||_1 \ge ||x_S||_1 - ||z_S||_1 + ||z_{\bar{S}}||_1.$$
  Since $||z_{\bar{S}}||_1 > ||z_S||_1$ we conclude $||x'||_1 > ||x||_1$.

  The argument is the same if we replace $A$ with $B$.
\end{proof}

\begin{thm}
  If $k>1$ divides $n$, $\exists x$ with $|\supp(x)| = \frac{n}{k}$,
  $Ax=y$ ($Bx=y$), and $x$ is not the unique solution to $Ax=y$ ($Bx=y$)
  minimizing $||x||_1$.
\end{thm}
\begin{proof}
  Let $x_i = 1$ when $i$ equals 1 modulo
  $k$ and $x_i = 0$ otherwise.  Let $z_i = 1$ when $i$ equals 2 modulo
  $k$ and $z_i = 0$ otherwise.  Then $Ax=Az=\mathbbm{1}$ ($Bx=Bz=\mathbbm{1}$), and $||x||_1 = ||z||_1$.
\end{proof}

\begin{obs} The previous theorem also holds replacing $||x||_1$ with $||x||_0$.
  This means the recovery guarantees we obtain with convex
  optimization are as good as an information theoretic bound.  That
  is, when we know $|\supp(x)| < n/k$ we can recover $x$ using convex
  optimization, but when $|\supp(x)| = n/k$ knowing the size of the
  support is not enough information to recover $x$.
\end{obs}
  
Related work: For the circular convolution case, Theorem 9 from
\cite{Donoho1989} also implies that when $x$ is sufficiently sparse we
can recover $x$ by minimizing $||x||_1$ subject to the constraints
defined by $y$.  When $k$ divides $n$ the Fourier transform of $h$ has
$k-1$ zeros.  In this case Theorem 9 from \cite{Donoho1989} implies we
can recover $x$ when $|\supp(x)| < \frac{n}{2(k-1)}$.  Here we give a
stronger result, showing the same optimization problem recovers $x$
when $|\supp(x)| < \frac{n}{k}$.  We also show our bound is tight and
that it matches the information theoretic limit in this case.  The
general results in \cite{Candes06} are also related to the problems
described here, but do not apply directly to our setting.

\bibliography{box}
\bibliographystyle{plain}

\end{document}